\documentclass[a4paper, 11pt, titlepage, fleqn]{article}
\usepackage{times}
\usepackage{amsthm}
\newtheorem{thm}{Theorem}
\newtheorem{cor}{Corollary}

\newtheorem{lemma}[thm]{Lemma}
\newtheorem{prop}{Proposition}

\usepackage{amssymb,amsmath}
\DeclareMathOperator{\ord}{ord}

\DeclareMathOperator{\F}{\mathbb{F}}

\begin{document}
	\baselineskip=16.3pt
	\parskip=14pt
	\begin{center}
		\section*{Slope Estimates for Generalized Artin-Schreier Curves}
		
		{\large 
			Gary McGuire\footnote{Corresponding Author, email gary.mcguire@ucd.ie, Research supported by Science Foundation Ireland Grant 13/IA/1914} and
			Emrah Sercan Y{\i}lmaz \footnote {Research supported by Science Foundation Ireland Grant 13/IA/1914} 
			\\
			School of Mathematics and Statistics\\
			University College Dublin\\
			Ireland}
	\end{center}
	\subsection*{Abstract}
	We provide first slope estimates 
	 of the Newton polygon of generalized Artin-Schreier curves, which are proved using
	the action of Frobenius and Verschiebung on cohomology.
	We provide a number of applications such as an improved Hasse-Weil bound
	for this class of curves.
	
	\subsection*{Keywords} Generalized Artin-Schreier, Supersingular, Newton Polygon.

	\section{Introduction}
	Let $q=p^u$ where $p$ is any prime and $u\ge 1$ is an integer, and
	let $Q=p^s$ where $s\ge 1$ is an integer.
	Let $X$ be a projective smooth absolutely irreducible curve of genus $g$ defined over $\mathbb{F}_Q$.
		Consider the $L$-polynomial of the curve $X$ over $\mathbb F_{Q}$, defined by
	$$L_X(T)=exp\left( \sum_{i=1}^\infty ( \#X(\mathbb F_{Q^i}) - Q^i - 1 )\frac{T^i}{i}  \right).$$
	where $\#X(\mathbb F_{Q^i})$ denotes the number of $\mathbb F_{Q^i}$-rational points of $X$. 
It is well known that $L_X(T)$ is a polynomial of degree $2g$ with integer coefficients, so we write it as 
\begin{equation} \label{L-poly}
L_X(T)= \sum_{i=0}^{2g} c_i T^i, \ c_i \in \mathbb Z.
\end{equation}

The Hasse-Weil bound places restrictions on the coefficients of $L_X(T)$ and on the values
of $\#X(\mathbb F_{Q^i})$. When we restrict ourselves to certain types of curve, such as supersingular curves,
or curves with Hasse-Witt invariant 0, even more restrictions are placed on 
$L_X(T)$ and $\#X(\mathbb F_{Q^i})$. This article is about generalized Artin-Schreier curves
and these restrictions.
Our main result is a bound on the first slope, see Theorem 1 below.

Consider the sequence of points 
$$\left\{ \left(i,\frac{\ord_p(c_i)}{s}\right): \ 0\leq i \leq 2g,  \right\}.$$
 in $\mathbb Q^2$. If $c_i=0$ for some $1\le i \le 2g$, we define $\ord_p(c_i)=\infty$.
The normalized $p$-adic Newton polygon of $L_X(T)$ is defined to be lower convex hull of this set of points. It is usually called the Newton polygon of $X/\mathbb F_Q$, and denoted by $NP(X/\mathbb F_Q)$. 
It is well known that $c_0=1$ and $c_{2g}=q^g$, so $(0,0)$ and $(2g,g)$ are
respectively the initial and the terminal points of the Newton
polygon.

	We call a curve $X$ a \emph{generalized Artin-Schreier} curve over $\mathbb F_Q$ if $X$
	 can be defined by an equation of the form
	\begin{equation} \label{curve-form}
	X\: : \:y^{q}-y=a_dx^d+a_{d-1}x^{d-1}+\cdots+a_1x+a_0
	\end{equation}
	where $a_i \in \mathbb{F}_Q$, $a_d \neq 0$, $(d,p)=1$, $q=p^u$ and $u$ is a positive integer.
We note that there is not necessarily any relationship between $u$ and $s$, where $Q=p^s$.
The genus of \eqref{curve-form} is $(q-1)(d-1)/2$.

In the case $q=p$ (i.e.\ $u=1$) the curve \eqref{curve-form}  is known as an Artin-Schreier curve.
In terms of function fields, an Artin-Schreier curve is a $p$-cyclic covering 
of the projective line over $\mathbb F_Q$ ramified only at infinity.
While there are many papers in the literature about Artin-Schreier curves, 
there are fewer about generalized Artin-Schreier curves, especially  in our context.
	
	Let $X/\mathbb F_Q$ be a generalized Artin-Schreier curve given by \eqref{curve-form}. 
	Define the support of $X$ by
	$$supp(X) \: :=\: \{ i \in \mathbb N \:|\: a_i \neq 0\}.$$ 
	Let $s_p(i)$ be the sum of all digits in the base $p$€ expansion of $i \in \mathbb{N}$.

	Let $NP_1(X/\mathbb F_Q)$ denote the first slope of $NP(X/\mathbb F_Q)$,
This is usually referred to as the first slope of $X/\mathbb F_Q$.
The first  and main result in this paper is the following.

\begin{thm}\label{weight-slope-thm}
	Let $X:y^{q}-y=f(x)$ where $f(x) \in \mathbb F_Q[x]$ has degree $d$, and let $\sigma=\max\{s_p(l) \: | \: l\in\text{supp}(X) \}$. Then $$NP_1(X/\mathbb F_{Q}) \ge \frac{1}{\sigma}.$$ 
\end{thm}

This theorem follows a line of results that one may call $p$-adic  bounds,
where the proofs rely on Stickelberger's theorem,
see for example Moreno and Moreno \cite{mm} and later  Blache \cite{RB}.
Our methods are completely different, and are similar to those of Scholten-Zhu in the papers \cite{SE}, \cite{FS} and \cite{HE}.

This paper is laid out as follows.
Section 2 gives important applications of Theorem  \ref{weight-slope-thm}.
In Section \ref{sse} we give some background 
 for the proof of Theorem  \ref{weight-slope-thm}.
Sections 4 and 5 develop this for generalized Artin-Schreier curves, and 
mostly follow the development of Scholten-Zhu
(no new ideas are needed in generalising from $p$ to $q$, however we include 
the proof for completeness).
In Section \ref{rtiling} we present new results for characteristic $p$ on $r$-tiling sequences.
Finally, Section \ref{pfthm1} presents the proof of Theorem \ref{weight-slope-thm},
and Section \ref{pfthm3} presents the proof of Theorem \ref{main-2}.

\section{Applications}

In this section we will give some applications of Theorem \ref{weight-slope-thm} 
where $X$ is a generalized Artin-Schreier curve of the form 
\eqref{curve-form}.

\subsection{First Slope}

Theorem \ref{weight-slope-thm} allows us to give a lower bound for the first slope of Newton polygon of the generalized Artin-Schreier curves \eqref{curve-form}
depending on $d$ and $p$, but not on $u$. The following corollary states this bound.

\begin{cor}\label{degree-slope-cor}
	Let $X:y^{p^u}-y=f(x)$ where $f(x) \in \mathbb F_Q[x]$ has degree $d$, and let \\
	$\tau=(p-1) \lceil \log_p(d) \rceil$. Then 
	\begin{equation}\label{tightbound}
	NP_1(X/\mathbb F_{Q}) \ge \frac{1}{\tau}. \end{equation}
\end{cor}

\begin{proof}
	We trivially have $\max\{s_p(l) \: | \: l\in\text{supp}(X) \} \le \lceil\log_pd\rceil(p-1)$. The statement now follows by Theorem \ref{weight-slope-thm}.
\end{proof}

This improves exponentially (compare $d/2$ with $\log (d+1)$)
on the bound in \cite{CX} where it is shown that
$NP_1(X/\mathbb F_{Q}) \ge \frac{1}{g}$
for the curve $y^2-y=f(x)$ in characteristic 2, where $f(x)$ has degree $d=2g+1$,
As remarked in \cite{HE}, the $1/g$ bound follows from properties of 
Newton polygons of abelian varieties.

The bound \eqref{tightbound} is tight in the sense that for each $p$ and $d$ there is a 
curve $X$ with first slope equal to $1/\tau$:
our proofs in the remainder of the paper show that the curves in Theorem \ref{main-2} give equality in \eqref{tightbound}.

\subsection{Divisibility}

We next relate the first slope to divisibility.
The $p$-divisibility of the coefficients in the L-polynomial \eqref{L-poly} is of interest for a few reasons. 
For example, Manin showed that the $p$-rank of the Jacobian of a curve (also known
as the Hasse-Witt invariant) is equal to 
the degree of the L-polynomial with coefficients reduced modulo $p$.
Thus, a curve has $p$-rank 0 precisely when all coefficients except the constant term
are divisible by $p$.
Mazur \cite{mazur} has drawn attention to the important problem of finding the $p$-adic valuations of the Frobenius eigenvalues,
which is closely related to the $p$-divisibility of the coefficients.

Another reason for studying the $p$-divisibility is to
prove supersingularity. The following is immediate from the
definition of supersingularity.

\begin{lemma}  \label{Sti-Xing}
A curve $X$ over $\F_Q$ (where $Q=p^s$) with the
$L$-polynomial $\eqref{L-poly}$ is
supersingular  if and only if
$$
\frac{\ord_p(c_i)}{s}\geq \frac{i}{2} \ \ \mbox{for all $i=1,\ldots ,2g$.}
$$
\end{lemma}

We will present a (partial) generalization of this for generalized Artin-Schreier curves
in Corollary \ref{GASdiv}, but first we need a simple proposition.


\begin{prop}\label{divisiblity-Sn-prop}
	Let $X:y^{p^u}-y=f(x)$ be a generalized Artin-Schreier curve, where $f(x) \in \mathbb F_Q[x]$ has degree $d$.
	Suppose that  $NP_1(X/\mathbb F_Q)=1/\sigma$ where $\sigma \ge 2$. Then: 
	\begin{enumerate} 
	\item If $X$ has L-polynomial $\eqref{L-poly}$, then $p^{\left\lceil si/\sigma\right\rceil}$ divides $c_i$ for $1\le i \le 2g$.
	\item $ p^{\left\lceil sn/\sigma\right\rceil}$ divides $|\#X(\mathbb F_{Q^n})-(Q^n+1)|$ for all integers $n\ge 1$. 
	\end{enumerate}
\end{prop}

\begin{proof}
	Let $X$ have L-polynomial $\eqref{L-poly}$. Since $NP_1(X/\mathbb F_Q)=1/\sigma$, 
	it follows from convexity of the Newton polygon that $p^{\left\lceil si/\sigma\right\rceil}$ divides $c_i$ where $1\le n \le 2g$.
	
	Let $S_n=|\#X(\mathbb F_{Q^n})-(Q^n+1)|$. Since we have $S_1=c_1$ is divisible by $p^{\left\lceil s/\sigma\right\rceil}$, and since we  have the well-known relation 
	$$c_1+2c_2t+\cdots+2gc_{2g}t^{2g-1}=(c_0+c_1t+\cdots+c_{2g}t^{2g})\sum_{n=1}^\infty S_nt^{r-1},$$ we get the result by induction.	
\end{proof}

It follows from Proposition  \ref{divisiblity-Sn-prop} part 1
 and the result of Manin mentioned above
that a generalized Artin-Schreier curve $X$
defined by \eqref{curve-form} has $p$-rank 0.
This is well known and can be proved by other methods (such as using
the Deuring-Shafarevich formula).

Here is the generalization of Lemma \ref{Sti-Xing}.

\begin{cor}\label{GASdiv}
A generalized Artin-Schreier curve $X$
defined by \eqref{curve-form}  over $\F_Q$ (where $Q=p^s$) with the
$L$-polynomial $\eqref{L-poly}$ 
and $NP_1(X/\mathbb F_Q)=1/\sigma$ has
$$
\frac{\ord_p(c_i)}{s}\geq \frac{i}{\sigma} \ \mbox{for all $i=1,\ldots ,2g$.}
$$
\end{cor}
\begin{proof}
This follows from  Proposition  \ref{divisiblity-Sn-prop} part 1.
\end{proof}

Next we state a simple corollary about the divisibility of
the trace of Frobenius.

\begin{cor}\label{degree-divisibility-cor}
	Let $X:y^{p^u}-y=f(x)$ where $f(x) \in \mathbb F_Q[x]$ has degree $d$ and let 
	$\tau=(p-1)\lceil\log_pd\rceil$. Then $ p^{\left\lceil sn/\tau\right\rceil}$ divides $|\#X(\mathbb F_{Q^n})-(Q^n+1)|$ for all $n\geq 1$. 
\end{cor}
\begin{proof}
	It is an easy consequence of Corollary \ref{degree-slope-cor} and Proposition \ref{divisiblity-Sn-prop}.
\end{proof}

\subsection{Improved Hasse-Weil Bound}

An improved Hasse-Weil bound is presented in \cite{CX} in characteristic 2,
for $Q=2^n$, $n$ odd.
We present a stronger improvement here, for any prime power $Q$,
which is strictly better for genus $>3$.

\begin{cor}\label{degree-bound-slope}
	Let $X:y^{p^u}-y=f(x)$ where $f(x) \in \mathbb F_Q[x]$ has degree $d$.
	Let $\tau=\lceil\log_pd\rceil(p-1)$. Then 
	$$ |\#X(\mathbb F_{Q^n})-(Q^n+1)|\le p^{\left\lceil sn/\tau \right\rceil}\left\lfloor \dfrac{g\lfloor 2\sqrt {Q^n} \rfloor}{p^{\left\lceil sn/\tau\right\rceil}}\right\rfloor$$ 
	where $g$ is the genus of $X$.
\end{cor}

\begin{proof}
	This corollary is a consequence of Corollary \ref{degree-divisibility-cor}. 
	The right hand side is the smallest integer which is divisible by $p^{\left\lceil sn/\tau \right\rceil}$ and smaller than the usual Hasse-Weil bound.
\end{proof}

Note that Corollary \ref{degree-divisibility-cor} may be more useful than Corollary \ref{degree-bound-slope} when considering families of curves, because using the divisibility property we can greatly reduce the number of possible values of $|\#X(\mathbb F_{Q^n})-(Q^n+1)|$ as $X$ ranges over the family. This can be useful when studying cyclic codes, see \cite{M} for example.

We give a few numerical examples to illustrate our results.

\textbf{Example 1:} 
Let $Q=2$, $n=7$ and $d=15$, then Hasse-Weil bound and the bound in \cite{CX} give $|\#X(\mathbb F_{2^7})-(2^7+1)| \le 154$  and the bound in Corollary \ref{degree-bound-slope}  gives $|\#X(\mathbb F_{2^7})-(2^7+1)| \le 152$. 

The divisibility property in \cite{CX} implies that 2 divides $|\#X(\mathbb F_{2^{7}})-(2^{7}+1)| $,  however the divisibility property in Corollary \ref{degree-bound-slope}  tells us that $4$ divides $|\#X(\mathbb F_{2^{7}})-(2^{7}+1)|$.

\textbf{Example 2:} 
Let $Q=2$, $n=101$ and $d=83$, then Hasse-Weil bound gives $|\#X(\mathbb F_{2^{101}})-(2^{101}+1)| \le 130565559286778326$, the bound in \cite{CX} gives $|\#X(\mathbb F_{2^{101}})-(2^{101}+1)| \le 130565559286778320$ (an improvement of 6), and the bound in Corollary \ref{degree-bound-slope}  
gives $|\#X(\mathbb F_{2^{101}})-(2^{101}+1)| \le 130565559286759424$ (an improvement of 18,902).

The divisibility property in \cite{CX} gives $ 2^3 \;\mid\; |\#X(\mathbb F_{2^{101}})-(2^{101}+1)| $  and the divisibility property in Corollary \ref{degree-bound-slope}  gives $2^{15}\; \mid \;|\#X(\mathbb F_{2^{101}})-(2^{101}+1)|$. 

\textbf{Example 3:} 
Let $Q=3$, $n=51$ and $d=104$, then Hasse-Weil bound gives $|\#X(\mathbb F_{3^{51}})-(3^{51}+1)| \le  302314665567277 $ and the bound in Corollary \ref{degree-bound-slope}  gives $|\#X(\mathbb F_{3^{51}})-(3^{51}+1)| \le 302314665566691$. 
The divisibility property in Corollary \ref{degree-bound-slope}  gives $3^{11}\; \mid \;|\#X(\mathbb F_{3^{51}})-(3^{51}+1)|$. 

\subsection{Family of Supersingular Curves}

A curve is said to be \emph{supersingular} if its Newton polygon is a straight line segment of slope $1/2$
(equivalently if $NP_1(X/\mathbb F_{Q}) =1/2$).
In van der Geer-van der Vlugt  \cite{VV} and Scholten-Zhu \cite{FS} it is shown that  all curves of the form 
	$$
	y^2-y=\sum\limits_{i=0}^{k}a_{2^i+1}x^{2^i+1}
	$$
	are supersingular over the finite fields having characteristic $2$. 
	In van der Geer-van der Vlugt  \cite{VV}, Blache \cite{GSE} and Bouw et al \cite{WMN}, 
	it is shown that for any prime $p$ all curves of the form
	$$
	y^p-y=\sum\limits_{i=0}^{k}a_{p^i+1}x^{p^i+1}
	$$
	are supersingular over  finite fields having characteristic $p$. In this paper we will generalize these results and prove the following theorem:

\begin{thm}\label{main} All curves of the form
	$$
	y^{q}-y=\sum\limits_{i,j=0}^{k}a_{p^i+p^j}x^{p^i+p^j}.
	$$
	are supersingular, where $a_{p^i+p^j} \in \mathbb{F}_Q$.
\end{thm}
Since $s_p(p^i+p^j)\le 2$ for all $i,j \ge 0$, Theorem \ref{main} follows from Theorem \ref{weight-slope-thm}.

\subsection{Family of Non-Supersingular Curves}

In the opposite direction, Sholten and Zhu showed in \cite{HE} that there is no hyperelliptic supersingular curve of genus $2^k-1$ in characteristic $2$ where $k \geq2$ (previously shown by Oort for genus 3). 
	Blache proved a similar result for all primes $p>2$ in \cite{RB} and showed that there is no supersingular Artin-Schreier  curve of genus $(p-1)(d-1)/2$ in characteristic $p$ where $n(p-1)>2$ and $d=i(p^n-1)$, $1 \leq i \leq p-1$.  We will generalize this result,
	and prove the following using the same techniques as we use to prove Theorem \ref{weight-slope-thm}.
	
	\begin{thm} \label{main-2}
		Let $d=i(p^n-1)$ with $n \geq 1$ and $1 \le i \le p-1$ and $n(p-1)>2$. Then 
		$$
		y^{p^u}-y=a_dx^d+a_{d-1}x^{d-1}+\cdots+a_1x+a_0
		$$
		is not supersingular for any $u\geq 1$.
	\end{thm}
	
	Putting $u=1$ recovers the result of Blache. The curves in Theorem \ref{main-2} have
genus $(p^u-1)(d-1)/2$, assuming $(p,d)=1$.

\subsection{Other Connections}

We remark that generalized Artin-Schreier curves have come up (see \cite{AGGMY}) 
in the completely different problem of studying irreducible polynomials over finite fields
with certain coefficients fixed.
A key part of the proof in \cite{AGGMY} is to calculate  the L-polynomials of
three specific generalized Artin-Schreier curves.


	\section{Sharp Slope Estimate for Arbitrary Curves}\label{sse}

	This section states a little background for the slope estimates of curves
	over finite fields. Note that Theorem \ref{lambda} and Lemma \ref{res} hold valid when the base field is perfect
	of characteristic $p$.
	
	Let $W$ be the Witt vectors over $\mathbb{F}_Q$, and $\sigma$ the absolute Frobenius automorphism of $W$. Throughout this paper we assume that $X/\mathbb{F}_Q$ is a curve of genus $g$ with a rational point. Suppose there is a smooth proper lifting $X/W$ of $X$ to $W$, together with a lifted rational point $P$. The Frobenius endomorphism  $F$ (resp., Verschiebung endomorphism $V$) are $\sigma$ ( resp., $\sigma^{-1}$) linear maps on the first crystalline cohomology  $H_{crys}^1(X/W)$ of $X$ with $VF=FV=p$. It is know that $H_{crys}^1(X/W)$ is canonically isomorphic to the first de Rham cohomology $H_{dR}^1(X/W)$ of $X$, one gets induced $F$ and $V$ actions on $H_{dR}^1(X/W)$.  Let $L$ be the image of $H^0(X,\Omega_{X/W}^1)$ in $H_{dR}^1(X/W)$.

	\begin{thm} \cite{SE} \label{lambda}
		Let $\lambda$ be a rational number with $0 \leq \lambda \leq 1/2$. Then $NP_1(X/\mathbb{F}_Q) \geq \lambda$ if and only if
		$$
		p^{\left \lceil{n\lambda}\right \rceil} \: | \: V^{n+g-1}L
		$$
		for all integer $n \geq 1$.
	\end{thm}
	
	Let $\hat{X}/W$ be formal completion of $X/W$ at rational point $P$. If $x$ is a local parameter of $P$, Then every element of $H_{dR}^1(\hat{X}/W)$ can be represented as $h(x)\frac{dx}{x}$ for some $h(x) \in xW[[x]]$, and $F$ and $V$ acts as follows:
	\begin{equation*}
	\begin{split}
	F\left(h(x)\frac{dx}{x}\right) &= ph^{\sigma}(x^p)\frac{dx}{x}\\
	V\left(h(x)\frac{dx}{x}\right) &= ph^{\sigma^{-1}}(x^{1/p})\frac{dx}{x} \quad \text{where 
	$x^{m/p}=0$ if $p \not| \: m$}
	\end{split}
	\end{equation*}
	
	Denote the restriction map $H_{dR}^1(\hat{X}/W) \rightarrow H_{dR}^1(X/W)$ by res. 
	
	\begin{lemma} \cite{SE} \label{res}
		The $F$ and $V$ action on $H_{dR}^1(X/W)$ and $H_{dR}^1(\hat{X}/W)$ commutes with the restriction map 
		$$
		res: \: H_{dR}^1(X/W) \quad \rightarrow \quad H_{dR}^1(\hat{X}/W).
		$$
		Furthermore,
		$$
		res^{-1}(H_{dR}^1(\hat{X}/W))=F(H_{dR}^1(X/W)).
		$$
	\end{lemma}
	
	\section{Slope Estimate of Generalized Artin-Schreier Curves}
	
Assume that $X$ is a curve over $\mathbb{F}_Q$ defined by an affine equation $y^q-y=\tilde{f}(x)$ where $q=p^u$ and  $\tilde{f}(x)=\tilde{a}_dx^d+\tilde{a}_{d-1} x^{d-1}+...+\tilde{a}_1x$ and $ p \not| \: d$ and $\tilde{a}_d \neq 0$. Take a lifting $X/W$ defined by $y^q-y=f(x)$ where $f(x)=a_dx^d+{a}_{d-1} x^{d-1}+...+a_1x \in W[x]$ with $a_l \equiv \tilde{a}_l$ mod $p$ for all $l$. So $X/W$ has a rational point at the origin with a local parameter $x$.\\
	For any integer $N>0$ and $0 \leq i \leq q-2$ let $C_r(i,N)$ be the $x^r$ coefficient of the power expansion of the function $y^i(qy^{q-1}-1)^{p^N-1}$ at the origin $P$: 
	$$
	y^i(qy^{q-1}-1)^{p^N-1}= \sum_{r=0}^{\infty}C_r(i,N)x^r.
	$$ 
	\begin{lemma} \label{basis}
		The curve $X/W$ has genus $(q-1)(d-1)/2$ and for $q-2 \geq i \geq 0$, $j \geq 1$ and $di+qj \leq (q-1)(d-1)-2+q$ the differential forms 
		$$
		\omega_{ij} := x^jy^i(qy^{q-1}-1)^{-1}\frac{dx}{x}
		$$
		form a basis for $N$.
	\end{lemma}
	\begin{proof}
		The proof of Lemma 3.1 in \cite{SE} stated for primes $p$ but is also valid for prime powers $q$.
	\end{proof}  
	
	\begin{lemma} \label{coprime} For $m$ be a positive integer. If $p \not | m$ then $x^m(qy^{q-1}-1)^{-1}\frac{dx}{x} \equiv 0 $ mod $q$ in $H_{dR}^1(\hat{X}/W)$.
	\end{lemma}
	\begin{proof}
		If $p \not | m$, then
		$$
		x^m(qy^{q-1}-1)^{-1}\frac{dx}{x} \equiv -x^m\frac{dx}{x} \equiv -d\bigg( \frac{x^m}{m} \bigg) \: \text{mod} \: q
		$$
		which is cohomologically zero in  $H_{dR}^1(\hat{X}/W)$.
	\end{proof} 
	
	\begin{lemma} \label{rel} For all nonnegative integer $a$ and $r$ we have
		$$
		C_r(i,N+a) \equiv C_r(i,N) \: \text{mod} \: p^{N+1}.
		$$
	\end{lemma}
	\begin{proof}
		We have $\binom{p^N}{l} \equiv 0$ mod $p^{N+1-l}$ if $N+1 \geq l \geq 1$. Thus
		$$
		(1-qy^{q-1})^{p^N}= \sum_{l=0}^{p^N}\binom{p^N}{l}(-qy^{q-1})^l \equiv 1 \: \text{mod} \: p^{N+1}.
		$$
		Therefore we have 
		$$
		y^i(qy^{q-1}-1)^{p^{N+a}-1}=	y^i(qy^{q-1}-1)^{p^{N}-1}(1-qy^{q-1})^{p^N(p^a-1)} \equiv 	y^i(qy^{q-1}-1)^{p^{N}-1}  \: \text{mod} \: p^{N+1}.
		$$
	\end{proof}
	
	\begin{thm}\label{keyT}
		Let $\lambda$ be a rational number with $0 \leq \lambda \leq 1/2$. Suppose there exists an integer $n_0$ such that\\
		(i) for all $i,j$ within the range $0 \leq i \leq q-2$, $j \geq 1$ and $di+pj \leq (q-1)(d-1)-2+q$ and for all $m \geq 1$, $1 \leq n < n_0$, we have
		$$
		\text{ord}_p \big ( C_{mp^{n+g-1}-j}(i,n+g-2)\big) \geq \left \lceil{n\lambda}\right \rceil ;
		$$
		(ii) for all $m \geq 2$ we have 
		$$
		\text{ord}_p \big ( C_{mp^{n_0+g-1}-j}(i,n+g-2)\big) \geq \left \lceil{n_0\lambda}\right \rceil .
		$$
		Then
		$$
		\begin{cases}
		p^{\left \lceil{n\lambda}\right \rceil} \: \mid \: V^{n+g-1}(\omega_{ij}) & \text{if } \: n <n_0 \\
		p^{\left \lceil{n_0\lambda}\right \rceil-1} \: \mid \: V^{n_0+g-1}(\omega_{ij}) & \text{if } \:  n =n_0. 
		\end{cases}
		$$
		Furthermore, we have 
		$$
		V^{n_0+g-1}(\omega_{ij})\equiv C^{\sigma^{-(n_0+g-1)}}_{p^{n_0+g-1}-j}(i,n_0+g-2)(\omega_{0,1}) \mod \left \lceil{n_0\lambda}\right \rceil .
		$$
	\end{thm}

	\begin{proof}
		We will prove by induction. Suppose $n\geq 1$ and 
			\begin{equation} \label{eq: div1}
		p^{\left \lceil{(n-1)\lambda}\right \rceil} \: | \: V^{n+g-2}(\omega_{ij}).
	\end{equation}
		Note this is trivially true if $n = 1$.\\
		Write $h(x)=(qy^{q-1}-1)^{-1} \in W[[x]]$. By (\cite{NN}, Lemma 2.2), we have
		$$
		h(x)^{p^{n+g-2}}=h^{\sigma^{n+g-2}}(x^{p^{n+g-2}})+ph_1^{\sigma^{n+g-3}}(x^{p^{n+g-3}})+...+p^{n+g-2}h_{n+g-2}(x)
		$$
		for some power series $h_1(x),h_2(x),...,h_{n+g-2}(x) \in W[[x]]$. Thus the power series expansion of $\omega_{ij}$ is
		\begin{equation*} \label{}
		\begin{split}
		res(\omega_{ij}) \quad &= \quad res\bigg(x^jy^i(qy^{q-1}-1)^{-1}\frac{dx}{x}\bigg)\\
		&= \quad res\bigg(x^jy^i(qy^{q-1}-1)^{p^{n+g-2}-1}h(x)^{p^{n+g-2}}\frac{dx}{x}\bigg)\\
		&= \quad \sum_{r=0}^{\infty}C_r(i,n+g-2)x^{r+j}h^{\sigma^{n+g-2}}\left(x^{p^{n+g-2}}\right)\frac{dx}{x}\\
		&\quad \quad +p\sum_{r=0}^{\infty}C_r(i,n+g-2)x^{r+j}h_1^{\sigma^{n+g-3}}\left(x^{p^{n+g-3}}\right)\frac{dx}{x}\\
		&\quad \quad +...\\
		&\quad \quad +p^{n+g-2}\sum_{r=0}^{\infty}C_r(i,n+g-2)x^{r+j}h_{n+g-2}(x)\frac{dx}{x}.
		\end{split}
		\end{equation*}
		Apply $V^{n+g-2}$ to the first differential form above. Since $V$ action commutes with the restriction map (by Lemma \ref{res}), we have
		\begin{equation} \label{resV-eq}
		\begin{split}
		res(V^{n+g-2}\omega_{ij}) &= \quad \sum_{m=1}^{\infty}C_{mp^{n+g-2}-j}^{\sigma^{-(n+g-2)}}(i,n+g-2)x^mh(x)\frac{dx}{x}\\
		&\quad \quad +p\sum_{m=1}^{\infty}C_{mp^{n+g-3}-j}^{\sigma^{-(n+g-2)}}(i,n+g-2)V\bigg(x^mh_1(x)\frac{dx}{x}\bigg)\\
		&\quad \quad +p^2\sum_{m=1}^{\infty}C_{mp^{n+g-4}-j}^{\sigma^{-(n+g-2)}}(i,n+g-2)V\left(x^mh_2(x)\frac{dx}{x}\right)\\
		&\quad \quad +...\\
		&\quad \quad +p^{\left \lceil{n\lambda}\right \rceil -1}\sum_{m=1}^{\infty}C_{mp^{n+g-1-\left \lceil{n\lambda}\right \rceil}-j}^{\sigma^{-(n+g-2)}}(i,n+g-2)V^{\left \lceil{n\lambda}\right \rceil -1}\bigg(x^mh_{\left \lceil{n\lambda}\right \rceil -1}(x)\frac{dx}{x}\bigg)\\
		&\quad \quad +p^{\left \lceil{n\lambda}\right \rceil}\beta
		\end{split}
		\end{equation}
		for some $\beta \in H_{dR}^1(\hat{X}/W)$.\\
		By the hypothesis, $p^{\left \lceil{n\lambda}\right \rceil -1}$ divides $C_{mp^{n+g-2}-j}(i,n+g-3)$. For all $m \geq 1$, by lemma \ref{rel},
		\begin{equation} \label{eq: div2}
			p^{\left \lceil{n\lambda}\right \rceil -1} \quad | \quad C_{mp^{n+g-2}-j}(i,n+g-2).
		\end{equation}
		For $m$ coprime to $p$ it follows from Lemma \ref{coprime} that $p$ divides $x^mh(x)\frac{dx}{x}$. Thus 
		$$
		p^{\left \lceil{n\lambda}\right \rceil} \quad | \quad C_{mp^{n+g-2}-j}(i,n+g-2)x^mh(x)\frac{dx}{x}.
		$$
		Otherwise, except possibly when $n=n_0$ and $m=p$, we have
		$$
		p^{\left \lceil{n\lambda}\right \rceil} \quad | \quad C_{(\frac{m}{p})p^{n+g-1}-j}(i,n+g-2)x^mh(x)\frac{dx}{x}.
		$$
		Therefore,
		
		\begin{align}	\label{eq: div3}
		&\sum_{m=1}^{\infty}C_{mp^{n+g-2}-j}^{\sigma^{-(n+g-2)}}(i,n+g-2)x^mh(x)\frac{dx}{x}\\  &\equiv \sum_{m'=1}^{\infty}C_{mp^{n+g-1}-j}^{\sigma^{-(n+g-2)}}(i,n+g-2)x^{pm'}h(x)\frac{dx}{x}\nonumber\\
		&\equiv 
		\begin{cases}
		0 \mod p^{\left \lceil{n\lambda}\right \rceil} &\text{if } n < n_0\\
		C_{p^{n_0+g-1}-j}^{\sigma^{-(n_0+g-2)}}(i,n+g-2)x^{p}h(x)\frac{dx}{x} \mod p^{\left \lceil{n\lambda}\right \rceil}  &\text{if } n = n_0. \nonumber
		\end{cases}
		\end{align}

		For all integers $l \geq 1$, by the hypothesis of the theorem, we obtain
		$$
		\text{ord}_p \big ( C_{mp^{n+g-l-2}-j}(i,n+g-l-3)\big) \geq \left \lceil{(n-1-l)\lambda}\right \rceil  \geq \left \lceil{n\lambda}\right \rceil -l .
		$$
		So by Lemma \ref{rel}, we have $
		\text{ord}_p \big ( C_{mp^{n+g-l-2}-j}(i,n+g-2)\big) \geq \left \lceil{n\lambda}\right \rceil -l .
		$ So  $p^{\left \lceil{n\lambda}\right \rceil}$ divides every sum of (\ref{resV-eq}) except possibly the one on the first line. Combining this information with (\ref{eq: div1}), (\ref{eq: div2}) and (\ref{eq: div3}) yields for all $n<n_0$ 
		$$
		res\bigg( \frac{V^{n+g-2}\omega_{ij}}{p^{\left \lceil{n\lambda}\right \rceil -1}} \bigg) \quad \in \quad pH_{dR}^1(\hat{X}/W).
		$$
		Hence for such $n$ Lemma \ref{res} implies
		$$
		\frac{V^{n+g-2}\omega_{ij}}{p^{\left \lceil{n\lambda}\right \rceil -1}} \quad \in \quad F(H_{dR}^1(X/W))
		$$
		so
		$$
		\frac{V^{n+g-1}\omega_{ij}}{p^{\left \lceil{n\lambda}\right \rceil -1}} \quad \in \quad VF(H_{dR}^1(X/W))=pH_{dR}^1(X/W)
		$$
		which proves the induction hypothesis. 
		
		If $n=n_0$ then the above implies that
		$$
		res\bigg( \frac{V^{n_0+g-2}\omega_{ij}}{p^{\left \lceil{n_0\lambda}\right \rceil -1}} \bigg)- \frac{1}{p^{\left \lceil{n_0\lambda}\right \rceil -1}}C_{p^{n_0+g-1}-j}^{\sigma^{-(n_0+g-2)}}(i,n+g-2)x^{p}h(x)\frac{dx}{x}
		$$
		lies in  $pH_{dR}^1(\hat{X}/W)$. Lemma \ref{res} implies
		$$
		\frac{V^{n_0+g-2}\omega_{ij}}{p^{\left \lceil{n_0\lambda}\right \rceil -1}} - \frac{1}{p^{\left \lceil{n_0\lambda}\right \rceil -1}}C_{p^{n_0+g-1}-j}^{\sigma^{-(n_0+g-2)}}(i,n+g-2)x^{p}\omega_{0,p}
		$$ 
		lies in $F(H_{dR}^1(X/W))$. Hence 
		$$
		\frac{V^{n_0+g-1}\omega_{ij}}{p^{\left \lceil{n_0\lambda}\right \rceil -1}} - \frac{1}{p^{\left \lceil{n_0\lambda}\right \rceil -1}}C_{p^{n_0+g-1}-j}^{\sigma^{-(n_0+g-1)}}(i,n+g-2)x^{p}V(\omega_{0,p})
		$$ 
		lies in $VF(H_{dR}^1(X/W))=p(H_{dR}^1(X/W))$.
	\end{proof} 
	
The next Lemma will be referred to as the Key Lemma.
	
	\begin{lemma} \label{key}
		Let $\lambda$ be a rational number with $0 \leq \lambda \leq \frac12$.\\
		
		(i) if for all $i,j$ within  range and for all $m \geq 1$, $n \geq 1$ we have
		$$
		\text{ord}_p \big ( C_{mp^{n+g-1}-j}(i,n+g-2)\big) \geq \left \lceil{n\lambda}\right \rceil 
		$$
		then
		$$
		NP_1(X/{\mathbb{F}_Q}) \geq \lambda.
		$$
		(ii) Let i,j be within range.\\
		(a) Let $n_0 \geq 1$.  Suppose that  and for all $m \geq 1$, $1 \leq n < n_0$ we have
		$$
		\text{ord}_p \big ( C_{mp^{n+g-1}-j}(i,n+g-2)\big) \geq \left \lceil{n\lambda}\right \rceil ;
		$$
		(b) suppose that for all $m \geq 2$ we have 
		$$
		\text{ord}_p \big ( C_{mp^{n_0+g-1}-j}(i,n+g-2)\big) \geq \left \lceil{n_0\lambda}\right \rceil;
		$$
		(c) suppose 
		$$
		\text{ord}_p \big ( C_{p^{n_0+g-1}-j}(i,n+g-2)\big) < \left \lceil{n_0\lambda}\right \rceil;
		$$
		then 
		$$
		NP_1(X/{\mathbb{F}_Q}) < \lambda.
		$$
	\end{lemma}
	\begin{proof}
		(i) The hypothesis in Theorem \ref{keyT} are satisfied for all positive integers $n_0$ and for all possible $i,j$. Thus the statement follows from Theorem \ref{lambda}.\\
		(ii) If $NP_1(X/\mathbb{F}_Q) \geq \lambda$ then $p^{\left \lceil{n_0\lambda}\right \rceil} \mid V^{n_0+g-1}(\omega_{ij})$ for all $i,j$ in the range of Theorem \ref{basis} by Theorem \ref{lambda}. This implies that for the particular $i,j$ satisfying the hypothesis of Theorem \ref{keyT} we have  
		$$
		\text{ord}\big( C_{p^{n_0+g-1}-j}(i,n+g-2)\big) < \left \lceil{n_0\lambda}\right \rceil.
		$$
		This proves the Lemma.
	\end{proof}
	
	We remark that if there is an decreasing sequence $\lambda_i$ whose limit 
	is $\lambda$, and all members $\lambda_i$ satisfy the Key Lemma Part 2, and if 
	$\lambda$ satisfies the Key Lemma Part 1, then $NP_1(X/\mathbb{F}_Q)=\lambda$.
	We will use this in the proof of Theorem \ref{main-2}.
	
	\section{p-adic Behavior Coefficients of Power Series}
	
	\begin{lemma}\label{ya} Let $a > 0$ and let $y \in W[[z]]$ be a power series that satisfies $y^q-y=z$ and $y(0)=0$. Then 
		$$
		y^a \: = \quad \sum_{k_1=0}^{\infty}D_{k_1}(a)z^{k_1}
		$$
		where $D_{k_1}(a)=0$ if $k_1 \not \equiv 0 \mod q-1$; otherwise,
		$$
		D_{k_1}(a)=(-1)^{a+\frac{k_1-a}{q-1}}\frac{a\left(k_1+ \frac{k_1 -a}{q-1} -1 \right)!}{k_1! \left(\frac{k_1-a}{q-1}\right)!}.
		$$
	\end{lemma}
	\begin{proof}
		The proof of Lemma 4.1 in \cite{SE} stated for primes $p$ but is also valid for prime powers $q$.
	\end{proof}
	\begin{lemma}\label{dk} Let $a > 0$  and  $k_1 \equiv a$  mod $q-1$, write $a=i+l(q-1)$  with integers $l$ and $1 \leq i \leq q-1$, then
		$$ 
		\begin{cases}
		ord_p(D_{k_1}(a)) =\frac{s_p(k_1)-s_p(i-1)-1}{p-1}     &\text{if } l=0 \\
		ord_p(D_{k_1}(a)) \ge\frac{s_p(k_1)-s_p(i-1)-1}{p-1} - (l-1)u   &\text{if } l \geq 1.
		\end{cases}
		$$
	\end{lemma}
	\begin{proof}
		$k_1 \equiv a$ mod $q-1$. Using the identity $(p-1)ord_p(k!)=k-s_p(k)$ for all positive integers $k$ we have 
		$$
		ord_p(D_{k_1}(a)) = ord_p(a)+ \frac{1}{p-1} \bigg(s_p(k_1)+s_p\bigg(\frac{k_1-a}{q-1}\bigg)-1-s_p\bigg(a-1+\frac{k_1-a}{q-1}q\bigg)\bigg).
		$$
		If $l=0$, then (note that $a=i$)
		$$
		s_p\bigg(a-1+\frac{k_1-a}{q-1}q\bigg)=s_p(a-1)+s_p\bigg(\frac{k_1-a}{q-1}\bigg).
		$$
		If $l=1$, then
		\begin{equation*} \label{}
		\begin{split}
		s_p\bigg(a-1+\frac{k_1-a}{q-1}q\bigg) &\leq s_p(a-1)+s_p\bigg(\frac{k_1-a}{q-1}\bigg)\\
		&=(p-1)ord_p(a)+(1+s_p(i-1))-1+s_p\bigg(\frac{k_1-a}{q-1}\bigg)\\
		&=(p-1)ord_p(a)+s_p(i-1)+s_p\bigg(\frac{k_1-a}{q-1}\bigg).
		\end{split}
		\end{equation*}
		If $l > 1$, then
		\begin{equation*} \label{}
		\begin{split}
		s_p\bigg(a-1+\frac{k_1-a}{q-1}q\bigg) &= 	s_p\bigg(i-1+l(q-1)+\frac{k_1-a}{q-1}q\bigg)\\
		&\leq s_p(i-1)+s_p(l(q-1))+s_p\bigg(\frac{k_1-a}{q-1}q\bigg)\\
		&\leq s_p(i-1)+(l-1)(p-1)u+s_p\bigg(\frac{k_1-a}{q-1}q\bigg).
		\end{split}
		\end{equation*}
	\end{proof}
	
	Fix two integers $N > 0$ and $0 \leq i \leq q-1$. Let $y \in W[[z]]$ be a power series that satisfies $y^q-y=z$ and $y(0)=0$. Define coefficients $E_{k_1}(i,N)$ by
	$$
	y^i(qy^{q-1}-1)^{p^N-1}= \sum_{k_1=0}^{\infty}E_{k_1}(i,N)z^{k_1}.
	$$
	Let $z=f(x)=a_1x+a_2x^2+\cdots+a_dx^d$. For ease of formulation, set $0^0 := 1$. Then 
	
	\begin{align*}
	\displaystyle\sum_{m=0}^\infty E_m(i,N) f(x)^m &=\displaystyle\sum_{m=0}^\infty E_m(i,N)  \left(a_1x+a_2x^2+\cdots+a_dx^d\right)^m\\ &=\displaystyle\sum_{m=0}^\infty E_m(i,N)  \displaystyle\sum_{\tiny\begin{matrix}m_1,m_2,\cdots m_d \geq 0\\m_1+m_2+\cdots+m_d=m\end{matrix}}\normalfont\left(\begin{matrix}
	m\\m_1,m_2,\cdots,m_d\end{matrix}\right)\displaystyle\prod_{l=1}^{d}\left(a_lx^l\right)^{m_l}.
	\end{align*}
	In order to find the coefficient of $x^r$ of $\displaystyle\sum_{m=0}^\infty E_m(i,N)  f(x)^m$, we have to find all $m_l$'s such that  $$\displaystyle\sum_{l=0}^dlm_l=r.$$
	
	Write $$m_i=k_i-k_{i+1} \text{ for } i=1,2,\cdots,d-1 \text{ and } m_d=k_d.$$ Since $m_i\ge 0$ for each $i=1,2,\cdots,d$, there is a one-to-one correspondence between $$(m_1,m_2,\cdots,m_d) \text{ such that } m_1+m_2+\cdots+m_d=m$$ and $$(k_1,k_2,\cdots,k_d) \text { such that } k_1 \ge k_2 \ge \cdots \ge k_d \ge 0 \text{ and } k_1+k_2+\cdots+k_d=r.$$ Moreover, we have $$m=\displaystyle\sum_{i=0}^dm_i=\displaystyle\sum_{j=0}^d(k_j-k_{j+1})=k_1$$ and $$\left(\begin{matrix}
	m\\m_1,m_2,\cdots,m_d\end{matrix}\right)=\left(\begin{matrix}
	k_1\\k_1-k_2,k_2-k_3,\cdots,k_{d-1}-k_d,k_d\end{matrix}\right)
	=\displaystyle\prod_{l=1}^{d-1}\left(\begin{matrix}
	k_l\\k_{l+1}\end{matrix}\right).$$
	For integers $r \geq 0$ let $\mathbf{K}_r$ denote the set of transposes $\mathbf{k}= \ ^{t}( k_1,k_2,...,k_d)$ of $d$-tuple integers with $k_1 \geq k_2 \geq ... \geq k_d \geq 0 $ and $ \sum_{l=1}^{d} k_l=r$. Moreover define $k_{d+1}=0$.
	
	Hence the $x^r$ of  $\displaystyle\sum_{m=0}^\infty E_m(i,N)  f(x)^m$ is $$\displaystyle\sum_{\mathbf{k} \in \mathbf{K}_r} E_{k_1}(i,N)\displaystyle\prod_{l=1}^{d}\left(\begin{matrix}
	k_l\\k_{l+1}\end{matrix}\right)a_l^{k_l-k_{l+1}}.$$
	
	Since $$\displaystyle\sum_{m=0}^\infty C_r(i,N)x^r=\displaystyle\sum_{m=0}^\infty E_m(i,N) f(x)^m,$$ we have $$C_r(i,N)=\displaystyle\sum_{\mathbf{k} \in \mathbf{K}_r} E_{k_1}(i,N)\displaystyle\prod_{l=1}^{d}\left(\begin{matrix}
	k_l\\k_{l+1}\end{matrix}\right)a_l^{k_l-k_{l+1}}.$$
	We define 
	$$
	s_p(\mathbf{k}) \: = \: s_p(k_1-k_2)+s_p(k_2-k_3)+...+s_p(k_{d-1}-k_d)+s_p(d).
	$$
	where	$\mathbf{k}= \ ^{t}( k_1,k_2,...,k_d) \in \mathbf{K}_r$ for some $r>0$.
	\begin{lemma}\label{ek} Let $\mathbf{k}= \ ^{t}( k_1,k_2,...,k_d) \in \mathbf{K}_r$. If $k_1 \not \equiv 0 \mod q-1$ then $E_{k_1}(i,N)=0$. If $k_1 \equiv$ mod $q-1$ then (define s(-1) := -1)
		\begin{equation*} 
		\begin{split}
		ord_p(E_{k_1}(i,N)) \quad &= \quad \frac{s_p(k_1)-s_p(i-1)-1}{p-1},\\
		ord_p\left(E_{k_1}(i,N)\prod_{l=1}^{d-1} \binom{k_l}{k_{l+1}}\right) \quad &= \quad \frac{s_p(\mathbf{k})-s_p(i-1)-1}{p-1}.
		\end{split}
		\end{equation*}
	\end{lemma}

	\begin{proof}
		Take the identity 
		$$
		y^i(qy^{q-1}-1)^{p^N-1}= \sum_{l=0}^{p^N-1}(-1)^{p^n-1-l}\binom{p^N-1}{l}q^ly^{i+l(q-1)}. 
		$$
		Subsitute the power series expansion of $y^{i+l(q-1)}$ above, we get
		$$
		E_{k_1}(i,N)= \sum_{l=0}^{p^N-1}(-1)^{p^n-1-l}\binom{p^N-1}{l}D_{k_1}(i+l(q-1))p^{ul}.
		$$
		$k_1 \not \equiv 0 \mod q-1$, then $D_{k_1}(i+l(q-1))=0$ by Lemma \ref{dk}, hence $E_{k_1}(i,N)=0$.\\
		If $k_1=i=0$, then $E_{k_1}(i,N)=(-1)^{p^N-1}$, hence $ord_p(E_{k_1}(i,N))=0$.\\
		If $i=0$ and $k_1 > 0$ and $k_1 \equiv$ mod $q-1$, by Lemma \ref{dk}, the term of minimal valuation occurs at $l=1$, we have
		$$
		ord_p(E_{k_1}(i,N))= u+ord_p(D_{k_1}(p-1))=u+\frac{s_p(k_1)-u(p-1)}{p-1}=\frac{s_p(k_1)}{p-1}.
		$$
		If $i >0$ and $k_1 \equiv$ mod $q-1$, by Lemma \ref{dk}, the term of minimal valuation occurs at $l=0$, we have
		$$
		ord_p(E_{k_1}(i,N))= u+ord_p(D_{k_1}(i))=\quad \frac{s_p(k_1)-s_p(i-1)-1}{p-1}.
		$$
		Moreover,
		$$
		ord_p \bigg(\prod_{l=1}^{d-1} \binom{k_l}{k_{l+1}}\bigg)=\sum_{l=1}^{d-1}ord_p\bigg(\binom{k_l}{k_{l+1}}\bigg)=\sum_{l=1}^{d-1}\frac{s_p(k_l)-s_p(k_{l+1})+s_p(k_l-k_{l+1})}{p-1}=\frac{s_p(\mathbf{k})-s_p(k_1)}{p-1}.
		$$
	\end{proof}
	
	\section{p-adic Boxes and r-tiling Sequences}\label{rtiling}
	
	Let $\mathbf{k}= \ ^{t}( k_1,k_2,...,k_d) \in \mathbf{K}_r$. We define integers $k_{l,v}$ as follows: For $l=d$ we let $k_d=\sum_{v \geq 0}k_{d,v}p^v$ be the $p$-ary expansion of $k_d$. For $1 \leq l \leq d$ we define $k_{l,v}$ inductively by
	$$
	k_{l,v} := k_{l+1,v}+p^v-\text{coefficient in the $p$-ary expansion of } (k_l-k_{l+1}),
	$$
	for all $v \geq 0$. We call the representation  $^{t}( k_1,k_2,...,k_d)$ the p-adic box of $\mathbf{k}$, denoted by $\boxed{\mathbf{k}}$ for short:
	$$
	\boxed{\mathbf{k}}=
	\begin{bmatrix}
	\cdots & k_{1,2} & k_{1,1} & k_{1,0} \\
	\cdots & k_{2,2} & k_{2,1} & k_{2,0} \\
	&   & \vdots &   \\
	\cdots & k_{d,2} & k_{d,1} & k_{d,0}
	\end{bmatrix}.
	$$
	
	Let $S$ be a finite set of positive integers. For any positive integers $r$, an r-tiling sequence (of length $v$) is a sequence of integer 3-tuples $\{[a_i,b_i,l_i] \}_{i=1}^v$ such that\\
	1) $l_i \in S$, $0 \leq b_i \leq b_{i+1}$, $1 \leq a_i \leq p-1$;\\
	2) $l_i > l_{i+1}$ if $b_i=b_{i+1}$;\\
	3) $\sum_{i=1}^{v}a_il_ip^{b_i}=r$.
	
	If no such sequence exists we set $\tilde{s}_p(r,S):=\infty$; otherwise, define  $\tilde{s}_p(r,S)$ the length of the shortest $r$-tiling sequence to be $\sum a_i$. Let $\tilde{K}(r,S)$ denote set of all shortest $r$-tiling sequences. 
	
	\begin{lemma}\label{tiling} For any positive integer $r$ and a finite set $S$ of positive integers,
		\begin{enumerate}
			\item if $\mathbf{k} \in \mathbf{K}_r$ with $k_l=k_{l+1}$ for all $l \not \in S$, then $s_p(\mathbf{k}) \geq \tilde{s}_p(r,S)$,
			\item  there is a bijection between the set $\tilde{K}(r,S)$ and the set
			$$
			\{\mathbf{k} \in \mathbf{K_r} | s_p(\mathbf{k}) =\tilde{s}_p(r,S) \: \text{and} \: k_l=k_{l+1} \: \text{for all} \: l \notin S \}.
			$$
		\end{enumerate}
		\end{lemma}
	\begin{proof}
		We shall define the maps first. An $r$-tiling sequence $\{[a_i,b_i,l_i] \}_{i=1}^{\tilde{s}_p(r,S)} \in \tilde{K}(r,S)$ is sent to the element $\mathbf{k} \in \mathbf{K}$ whose $p$-adic box $\boxed{\mathbf{k}}$ has $k_{l,v}\: = \sum_{i \in U} a_i$ where $U=\{j | v=b_j \: \text{and} \: l \leq l_j \}$.\\
	Given $\mathbf{k} \in \mathbf{K}$ with $k_l=k_{l+1}$ for all $l \not \in S$, one defines $\{[a_i,b_i,l_i] \}_{i=1}^{s(\mathbf{k})}$ as follows: Given $b$ let $l$ be largest value such that $k_{l,b}=a$ is nonzero.  Then we get a 3-tuple $[a,b,l]$. Subtract $a$ from each component $k_{l',b}$ with $1 \leq l' \leq l$, then apply the same procedure if there is a nonzero element. Note that: $l$ is in $S$ by definition, since it only change when $l \in S$; and $a \in \{1,...,p-1 \}$ since $a$ is $p^v$-th coefficient in the "base $p$" expansion of $(k_l-k_{l+1})$.\\
		These the maps are well-defined and one-to-one. Since the sets are finite, the maps are bijective.
	\end{proof}

		\section{Proof of Theorem \ref{weight-slope-thm} }\label{pfthm1}

	\begin{proof}[Proof of Theorem \ref{weight-slope-thm}:]
		Recall
		$$
		C_r(i,N)=\sum_{k \in K_r} E_{k_1}(i,N) \prod_{l=1}^{d-1} \binom{k_l}{k_{l+1}} \prod_{l=1}^{d-1} a_l^{k_l-k_{l+1}}.
		$$
		Note that $\prod\limits_{l=1}^{d-1}a_l^{k_l-k_{l+1}}=0$ if $k_l>k_{l+1}$ for some $l \not \in s(X)$ and if $k_l=k_{l+1}$ for all $l \not \in s(X)$ then  $s_p(\mathbf k) \geq \tilde s_p(r,s(X))$ where $\mathbf k \in \mathbf K_r$ by Lemma \ref{tiling} Part 1. Therefore, for $r$ in $\{mp^{n+g-1}-j |m,n \geq 1, \text{ and $j$ within range} \}$, we have
		$$
		ord_p(C_r(i,N)) \geq  \frac{\tilde s_p(r,s(X))-s_p(i-1)-1}{p-1}.$$
		
		Moreover, for $r=mp^{n+g-1}-j$ consider any $r$-tiling sequence given by $mp^{n+g-1}-j=\sum_{i =0}^{v}a_il_ip^{b_i}$, then we have
		$$
		\sigma\tilde{s}_p(r,s(X))=\sum_{i =0}^{v}\sigma a_i  \geq \sum_{i =0}^{v}s_p(a_il_i) = \sum_{i =0}^{v}s_p(a_il_ip^{b_i}) \geq s_p(\sum_{i =0}^{v}a_il_ip^{b_i})=s_p(r).
		$$
		
		Let $p$ be an odd prime and $d>2$ be a positive integer. Note that the case $d=2$ is easy, so we can restrict it as $d>2$. We have $ 0\le i \le p-2$ and $1 \le j <d(p-1)-1$. Therefore, $$\tilde{s}_p(r,X) \ge \frac1\sigma\left[(p-1)\left(n+g-1-\sigma\right)+1\right]$$ and hence \begin{align*}
		C_r(i,N) &\ge \frac{\tilde{s}_p(r,X) -s_p(i)}{p-1}\\
		&\ge \dfrac{\frac1\sigma\left[(p-1)\left(n+g-1-\sigma\right)+1\right]-(p-2)}{p-1}\\ & =\frac{n}{\sigma}+\frac{g+2-p-\sigma}{\sigma}\\
		&\ge \frac n{\sigma}.\end{align*}
		
		Let $q=p^u$ be a prime power with $p$ prime, $u\ge2$ integer and $d$ positive integer. We have $ 0\le i \le p^u-2$ and $1 \le j <d(q-1)-1$. Therefore, $$\tilde{s}_p(r,X) \ge \frac1\sigma(p-1)\left(n+g-1-\sigma\right)$$ and hence \begin{align*}
		C_r(i,N) &\ge \frac{\tilde{s}_p(r,X) -s_p(i-1)-1}{p-1}\\
		&\ge \dfrac{\frac1\sigma(p-1)\left(n+g-1-\sigma\right)-u(p-1)}{p-1}\\ 
		& =\frac{n}{\sigma}+\frac{g-1-\sigma-u}{\sigma}\\
		&\ge\frac n{\sigma}\end{align*}
		except for $(q,d) \ne(4,3)$. This easy case is also fine when we specially optimize the upper bounds of $s_p(i)$ and $s_p(j)$. We omit the details.
	\end{proof}

		\section{Proof of Theorem  \ref{main-2}}\label{pfthm3}
	In this section, we will prove Lemma \ref{kbox} and Lemma \ref{C-mod-p^k+!}  and then prove Theorem \ref{main-2}. Let $d=j(p^h-1)$ with $h \geq 1$ and $1\le j \le p-1$.

	\begin{lemma} \label{kbox}
		Let $\mathbf k \in \mathbf K_r$ with $r \geq 1$.
		\begin{enumerate}
			\item We have $s_p(\mathbf k) \geq \left\lceil \frac{s_p(r)}{h(p-1)} \right\rceil$.
			\item  If $s_p(\mathbf k)= \left\lceil \frac{s_p(r)}{h(p-1)} \right\rceil$ then 
			\begin{enumerate}
				\item the p-adic boxes $\boxed{\mathbf k}$ consists of only $0$ or $1$; or $2$ when $\frac{p^2-1}2 \leq d < p^2-1$.
				\item 	for every $v \geq 0$, $s_p(\sum_{l=1}^{d}k_{l,v})=k_{1,v}h(p-1).$
				\item  For any $r=j(p^h-1)$ with $1 \le j \le p-1$ and $h \ge 1$, the p-adic boxes $\boxed{\mathbf k}$ consists of only $0$ or $1$.
			\end{enumerate}
		\end{enumerate} 
		
	\end{lemma}
	\begin{proof}
		Let $k$ be a nonnegative integer and note that $d \ge p-1$. Then we have
		\begin{align*}
		\lfloor \log_p(kd+1)\rfloor &\leq \lfloor \log_p(kd+d)\rfloor \\ &= \lfloor \log_p\left(k(d+1)\right)\rfloor \\ &= \lfloor \log_pk+\log_p(d+1)\rfloor \\
		&\leq \lfloor (k-1)+\log_p(d+1)\rfloor \\ 
		&=(k-1)+ \lfloor\log_p(d+1)\rfloor\\
		&\leq (k-1)\lfloor \log_p(d+1)\rfloor+\lfloor \log_p(d+1)\rfloor\\ 
		&\leq k\lfloor \log_p(d+1)\rfloor.
		\end{align*} 
		Therefore, for any nonnegative integer $k$ and any degree $d$ greater than $p-2$ the inequality $$\lfloor \log_p(kd+1)\rfloor \le k\lfloor \log_p(d+1)\rfloor$$ holds and it is obvious that equality holds for $k=0$ or $k=1$. \\
		Now assume the equality holds, the above system tells us $\lfloor \log_p(d+1)\rfloor$ must equal to $1$. The equation $\lfloor \log_p(d+1)\rfloor=1$ holds if and only if $p-1 \leq d < p^2-1$. Moreover, if $p-1 \leq d < p^2-1$ and $\lfloor \log_p(kd+1)\rfloor=k$, we have
		\begin{align*}
		p-1 \leq d < p^2-1 &\implies kp-(k-1) \leq kd+1 < kp^2-(k-1)\\
		&\implies p^k<kp^2-(k+1). 
		\end{align*}
		The last inequality holds only for $k \le 2$. Now assume $k=2$. Then $$p\leq d+1 < p^2 \leq 2d+1 <p^3$$ $$\implies \frac{p^2-1}2 \leq d < p^2-1.$$
		Since $r= \displaystyle\sum\limits_{v \geq 0}\sum_{l=1}^dk_{l,v}p^v$, we have 
		\begin{align}
		s_p(r) &= s_p\left( \displaystyle\sum\limits_{v \geq 0}\sum_{l=1}^dk_{l,v}p^v\right) \nonumber \\&= \displaystyle\sum\limits_{v \geq 0}s_p\left(\sum_{l=1}^dk_{l,v}p^v\right) \nonumber\\&\leq \sum\limits_{v\geq0}s_p\left(\sum\limits_{l=1}^dk_{l,v}\right) \label{a} \\&\leq \sum_{v\geq0} (p-1)\left\lfloor\log_p\left(\sum\limits_{l=1}^dk_{l,v}+1\right)\right\rfloor\nonumber\\
		&\leq (p-1)\sum_{v\geq0}\left\lfloor \log_p(k_{1,v}d+1)\right\rfloor \nonumber\\&\leq(p-1)\left(\sum_{v\geq0}k_{1,v}\lfloor\log_p(d+1)\rfloor\right)\nonumber\\&=\lfloor\log_p(d+1)\rfloor(p-1)\sum_{v\geq0}k_{1,v} \label{b}\\ &=h(p-1)s_p(\mathbf k).\nonumber
		\end{align}
		Thus we have $s_p(\mathbf k) \geq \left\lceil \frac{s_p(r)}{h(p-1)}\right\rceil.$ This proves the first assertion.\\
		
		Suppose above equality holds then $k_{1,v}=0$ or $1$; or $2$ when $\frac{p^2-1}2 \leq d < p^2-1$. Since $k_{l,v} \geq k_{l+1,v}$, the first observation in the second assertion is proved.\\
		
		By the arguments we used from equation \eqref{a} to equation \eqref{b}, we have $s_p\left(\sum\limits_{l=1}^dk_{l,v}\right)= k_{1,v}(p-1)\lfloor \log_p(d+1)\rfloor=k_{1,v}h(p-1)$. This proves the second observation in the second assertion.\\
		
		Let $d \le (p-1)^2$. Then $2d \le 2(p-1)^2$ but the only value $a \le 2d$ such that $s_p(a)=2(p-1)$ can be $p^2-1$. Since $p^2-1<j(p^2-1)$ for all $j$ such that $j \in [\frac{p+1}2,p-1]$, trivially, for any $j(p^h-1)$ with $1 \le j \le p-1$ and $h \ge 1$ the p-adic boxes $\boxed{\mathbf k}$ consists of only $0$ or $1$. This proves the third  observation in the second assertion.
	\end{proof}
	\begin{lemma} \label{C-mod-p^k+!}
		Let notation be as above. $C_{j(p^{bh}-1)}(i,N) \equiv p^b(a_{j(p^h-1)})^\frac{p^{bh}-1}{p^h-1} \mod p^{b+1}.$
	\end{lemma}
	\begin{proof}
		Let $r_{j,b}=j(p^{bh}-1)$ where $0 \leq j \leq p-1$ and $b\geq1$. Let $\mathbf k \in \mathbf K_{r_{j,b}}$ then we have $s_p(\mathbf k) \geq \bigg\lceil \frac{s_p(j(p^{bh}-1))}{h(p-1)}\bigg\rceil=\bigg\lceil \frac{hb(p-1)}{h(p-1)}\bigg\rceil=b$. Now suppose $s_p(\mathbf k)=b$. Let $\gamma_t(\mathbf{k})$ be the sum of entries in the $t$-th none-zero column (from left). Recall that $\boxed {\mathbf {k}}$ has $d$ rows. By Lemma \ref{kbox}  we know that $\boxed {\mathbf k}$ consists of only $0$ or $1$. Therefore we have $\gamma_t(\mathbf k)\leq d$. 
		For $d=i(p^h-1)$ we have $s_p(\gamma_t(\mathbf k))=h(p-1)$ for all such $t$. This only can happen when all the entries of the $t$-th column is $1$. Hence we have all $0$ columns or all $1$ columns. Since $j(p^{bh}-1)=[j(p^h-1)](p^{h(b-1)}+p^{h(b-2)}\cdots+1)$. The only possibility is that $\gamma_t(k)=j(p^h-1)$ for all such $t$.
		It is clear that $\mathbf k=\: ^t(k_1,\cdots,k_d)\in \mathbf K_{r_{j,b}}$ is defined by $k_1=\cdots=k_{j(p^h-1)}=\frac{r_{j,b}}{j(p^h-1)}$. Thus we have 
		$$
		C_{r_{j,b}}(i,N)\equiv p^b(a_{j(p^h-1)})^{\frac{r_{j,b}}{j(p^h-1)}}\mod p^{b+1}.
		$$
	\end{proof}
	
	\begin{proof}[Proof of Theorem \ref{main-2}]
		From Lemma \ref{kbox} it follows that 
		\begin{align*}
		\text{ord}_p \left(C_{mp^{n+g-1}-j}(i,N)\right) &\ge \dfrac{s_p(\mathbf{k})-s_p(i-1)-1}{p-1}\\ &\ge  \dfrac{s_p(\mathbf{k})-u(p-1)}{p-1} \\&\ge  \dfrac{\left\lceil\dfrac{s_p(mp^{n+g-1}-j)}{h(p-1)}\right\rceil-s_p(i-1)-1}{p-1}\\ & \ge \left\lceil\dfrac{\dfrac{(p-1)(n+g-h)}{h(p-1)}-u(p-1)}{p-1}\right\rceil\\&=\left\lceil\dfrac{(n+g-h)-uh(p-1)}{h(p-1)}\right\rceil \\&\ge \left\lceil\dfrac{n}{h(p-1)}\right\rceil
		\end{align*}
		for all $m,n \geq 1$. By Lemma \ref{key}-a we have $NP_1(X)\geq \frac{1}{h(p-1)}$.
		
		From now on assume $NP_1(X)>\frac{1}{h(p-1)}$. For any integer $n>1$ define 
		$$
		\lambda_n := \frac{n+g-2-uh(p-1)}{h(p-1)(n-1)}.
		$$
		
		We now apply the observation made in the remark after the Key Lemma (Lemma \ref{key}).
		Consider $\lambda_n$ as a function in $n$, it is clear that $\lambda_n$ is monotonically decreasing and converges to $\frac{1}{h(p-1)}$ as $n$ approaches $\infty$. Choose $n_0$ such that $\lambda_{n_0}\leq NP_1(X)$ and such that $n_0+g-1$ is a multiple of $h(p-1)$ and $\frac{g-1-uh(p-1)}{h(p-1)(n_0-1)}\leq 1$. For all $1 \leq n <n_0$ we have $\lambda_{n_0}\leq \lambda_{n+1}=\frac{n+g-1-uh(p-1)}{nh(p-1)}$; that is,
		$$
		\lceil n\lambda_{n_0} \rceil \leq \bigg\lceil\frac{n+g-1-uh(p-1)}{h(p-1)}\bigg\rceil.
		$$
		Therefore, for all $m\geq 1$ and $1\leq n <n_0$ one has 
		$$
		ord_p(C_{mp^{n+g-1}-j}(i,N)) \geq \bigg\lceil\frac{n+g-1-uh(p-1)}{h(p-1)}\bigg\rceil \geq \lceil n\lambda_{n_0} \rceil.
		$$
		On the other hand, 
		\begin{align*}
		\left\lceil n_0\lambda_0\right\rceil&=\left\lceil n_0\cdot\dfrac{n_0+g-2-uh(p-1)}{h(p-1)(n_0-1)}\right\rceil\\&=\left\lceil \dfrac{(n_0-1)+g-1-uh(p-1)}{h(p-1)}\left(1+\frac{1}{n_0-1}\right)\right\rceil\\ &=\left\lceil\dfrac{n_0+g-1-uh(p-1)}{h(p-1)}+\dfrac{g-1-uh(p-1)}{h(p-1)(n_0-1)}\right\rceil\\&=\dfrac{n_0+g-1}{h(p-1)}-u+\left\lceil\dfrac{g-1-uh(p-1)}{h(p-1)(n_0-1)}\right\rceil\\&=\dfrac{n_0+g-1}{h(p-1)}-u+1.
		\end{align*}
		Hence for all $m\geq 2$  one has
		$$
		ord_p(C_{mp^{n_0+g-1}-j}(i,N))\geq \bigg\lceil\frac{n_0+g}{h(p-1)}-u\bigg\rceil=\frac{n_0+g-1}{h(p-1)}-u+1=\lceil n_0\lambda_{n_0} \rceil.
		$$
		Thus the hypotheses of Lemma \ref{key}-b are satisfied (for $j=1$ and $\lambda=\lambda_{n_0}$ too) so we 
		$$
		ord_p(C_{j(p^{n_0+g-1}-1)}(i,N) \geq \lceil n_0\lambda_{n_0} \rceil.
		$$
		for $1\le j \le p-1$. We also have
		$$
		ord_p(C_{j(p^{n_0+g-1}-1)}(i,N)\equiv p^{ \lceil n_0\lambda_{n_0} \rceil-1}a_{j(p^n-1)}^{({p^{n_0+g-1}-1})/{(p^h-1)}} \mod p^{ \lceil n_0\lambda_{n_0} \rceil}.
		$$
		Hence $a_{j(p^n-1)}=0 \mod p$. Thus $c_{j(p^n-1)}=0$.
	\end{proof}


\begin{thebibliography}{99}
		
		
		
		
				
		\bibitem{AGGMY} O. Ahmadi, F. Gologlu, R. Granger, G. McGuire, E. S. Yilmaz, Fibre Products of Supersingular Curves and the Enumeration of Irreducible Polynomials with Prescribed Coefficients, Finite Fields and Their Applications,
		vol 42 (2016) 128--164.
		
		\bibitem{GSE} R. Blache, Valuation of exponential sums and the generic first slope for Artin–Schreier curves. Journal of Number Theory, 132:2336–2352, 2012.
		
		\bibitem{RB} R. Blache, Valuations of exponential sums and Artin-Schreier curves ,arXiv:1502.00969 

		
		\bibitem{WMN} I. Bouw, W. Ho, B. Malmskog, R. Scheidler, P. Srinivasan, and C. Vincent, Zeta functions of a class of Artin-Schreier curves with many automorphisms, 2014.
		
		\bibitem{CX} R. Cramer, C. Xing, An improvement to the Hasse-Weil bound and applications to character sums, cryptography and coding, Advances in Mathematics, Vol 309,  March 2017, Pages 238-253
		
		\bibitem{VV} G. van der Geer, M. van der Vlugt, Reed-Muller Codes and supersingular curves I, Compositio Mathematica, 84, No. 3 (1992) 333-367.

		\bibitem{mazur} B. Mazur, Frobenius and the Hodge filtration, Bulletin of the AMS,
		vol 78, No 5, Sept. 1972, 653--667.
		
		\bibitem{M}  G. McGuire, An alternative proof of a result on the weight divisibility of a cyclic code using supersingular curves, Finite Fields Appl. 18 (2012) no.2, 434-436
		
		\bibitem{mm} O. Moreno and C. J. Moreno,  {\em An elementary proof of a partial improvement to the Ax-Katz Theorem}; Applied algebra, algebraic algorithms and error-correcting codes (San Juan, PR, 1993), 257--268, Lecture Notes in Comput. Sci., 673, Springer, Berlin, 1993.


		
		\bibitem{NN} N. Nygaard, On supersingular abelian varieties. Algebraic geometry (Ann Arbor, Mich., 1981), 83-101.

		\bibitem{SE} J. Scholten, H. J. Zhu, Slope Estimates of Artin-Schreier Curves, Compositio Mathematica
137:
275--292, 2003.
		
		\bibitem{FS} J. Scholten, H. J. Zhu, Families of supersingular curves in characteristic 2, Math. Research Letters 9, no 5-6, (2002) 639--650.
		
		\bibitem{HE} J. Scholten, H. J. Zhu: Hyperelliptic curves in characteristic 2. Inter. Math. Research Notices. 17 (2002), 905–917.
		
				
		
		
	\end{thebibliography}
\end{document}